\documentclass{article}
\usepackage{amsmath,amssymb,verbatim,amsthm}
\usepackage{amsopn}
\usepackage[mathscr]{eucal}
\usepackage{graphicx}
\usepackage{color}
\usepackage[normalem]{ulem}
\usepackage{enumerate}

 \newtheorem{thm}{Theorem}[section]
 \newtheorem{cor}[thm]{Corollary}
 \newtheorem{lem}[thm]{Lemma}
 
 \theoremstyle{definition}
 
 \newtheorem{defn}[thm]{Definition}
 \theoremstyle{remark}

 \numberwithin{equation}{section}
 \newtheorem{prethought}[thm]{Thought}
 \newtheorem{needsfixing}{Needsfixing}

\begin{document}
\title{Minding's Theorem for Low Degrees of Differentiability}
\author{Josef F. Dorfmeister \thanks{Zentrum Mathematik, Technische Universit\"{a}t M\"{u}nchen, D-85747 Garching bei M\"{u}nchen, Germany, dorfm@ma.tum.de} \and 
Ivan Sterling \thanks{Mathematics and Computer Science Department, St Mary's College of Maryland, St Mary's City, MD 20686-3001, USA, isterling@smcm.edu}}
\maketitle
\begin{abstract}
\noindent We prove Minding's Theorem for $C^2$-immersions with constant negative Gauss curvature.  As a Corollary we also prove Minding's Theorem for $C^{1M}$-immersions in the sense of \cite{DS}.
\end{abstract}
\section{Statement of Main Theorem}
Let $\Omega$ denote a simply connected open set in $\mathbb{R}^2$.  We say that a function $f:\Omega \longrightarrow \mathbb{R}^\ell$ is of class $C^k(\Omega)$ (written $f \in C^k(\Omega)$) if all its derivatives up to order $k$ are continuous.  $k=\infty$ ($k=\omega)$ if the function is infinitely differentiable (resp. analytic).  A coordinate chart on $\Omega$ is a pair $(\beta_{(t^1,t^2)},U)$ where $U$ is open in $\Omega$ and $\beta_{(t^1,t^2)}:U \longrightarrow V \subset \mathbb{R}^2_{(t^1,t^2)}$ (always assumed to be at least $C^1(U)$ with non-vanishing Jacobian).  A metric (always assumed to be positive definite) $g = \sum g_{ij} dt^i dt^j$ is said to be of class $C^k(V)$ in the coordinate chart $(\beta_{(t^1,t^2)},U)$ if the coefficients $g_{ij} \in C^k(V)$. 
A coordinate chart $(\beta_{(x,y)},U)$ is called isothermic with respect to a metric $g=g_{11} dx^2+g_{12} dxdy + g_{21} dydx + g_{22} dy^2$  if $g_{11}\!=\!g_{22}$ and $g_{12}(=\!g_{21})\!=\!0$.  If $(\beta_{(x,y)},U)$ is isothermic then $h^2:=g_{11}\!=\!g_{22}$ is called the conformal factor of $g=h^2(dx^2+dy^2)$ with respect to $(\beta_{(x,y)},U)$.

A $C^n(\Omega)$, $n \geq 1$, map $f:\Omega \longrightarrow \mathbb{R}^\ell$ is called an immersion (always assumed to be regular) if $rank(df) \equiv 2$.  The $C^{n-1}(\Omega)$-metric, induced on $\Omega$ by $f$ is given by $f^*ds^2_{\mathbb{R}^\ell}$, where $ds^2_{\mathbb{R}^\ell}$ denotes the standard metric on $\mathbb{R}^\ell$.  If $\Omega_1$ has a metric $g_1$ and $\Omega_2$ has a metric $g_2$ then a $C^1(\Omega_1)$-immersion $\phi: (\Omega_1,g_1) \longrightarrow (\Omega_2,g_2)$ is called an isometry if $\phi^*(g_2) = g_1$.  We will prove the following case of Minding's Theorem, which was previously known only for $n \geq 3$.
\begin{thm}\label{mainthm}
Let $f:\Omega \stackrel{C^n}\longrightarrow \mathbb{R}^3$ ($n \geq 2$) be a $C^n(\Omega)$-immersion with the induced $C^{n-1}(\Omega)$-metric $f^*ds^2_{\mathbb{R}^3}$ and  Gauss curvature $K \equiv -1$.  Then there exists a $C^2(\Omega)$ isometry $\phi:(\Omega,f^*ds^2_{\mathbb{R}^3}) \longrightarrow (\mathbb{H}^2, ds^2_{\mathbb{H}^2})$,  where $ds^2_{\mathbb{H}^2}$ denotes the standard metric on $\mathbb{H}^2$.
\end{thm}
\noindent From Theorem \ref{dsthm} in Section \ref{seccor}, proved in \cite{DS}, we immediately have the following Corollary \ref{maincor}.  The $C^{1M}$ surfaces (see definition in Section \ref{seccor}) which appear in Corollary \ref{maincor} arise naturally in the loop group classification of $K \equiv -1$ surfaces \cite{DS}.  For a further explanation of the terms used in Corollary \ref{maincor} see Section \ref{seccor}.
\begin{cor} \label{maincor}
Let $f=f_{asyche}:\Omega \stackrel{C^{1M}}{\longrightarrow} \mathbb{R}^3$ be a $C^{1M}$-immersion in asymptotic Chebyshev coordinates.  Assume  $N=N_{asyche}:= \frac{f_x \times f_y}{\sin{\theta}}$ is $C^{1M}$, $N_{xy}=N_{yx}=\cos \theta N$, $f_x = N \times N_x$, and $f_y=-N \times N_y$. (Note this implies  $K \equiv -1$.)   Then there exists a $C^2(\Omega)$ isometry $\phi:(\Omega,f^*ds^2_{\mathbb{R}^3}) \longrightarrow (\mathbb{H}^2, ds^2_{\mathbb{H}^2})$.
\end{cor}
\section{Preliminary Results}
The proof of Theorem \ref{mainthm} will be as follows.  By Theorem \ref{thmchw}, is suffices to prove the theorem assuming isothermic coordinates, with a metric $h^2 (dx^2+dy^2)$ and $h \in C^1$.  Using distributional derivatives we will prove that $u=\ln h$ is a weak solution to the equation $\Delta u = e^{2u}$.  Using bootstrap results from elliptic PDE theory we will first show that $u \in C^2(\Omega)$ and then that $u \in C^\infty(\Omega)$.  Theorem \ref{mainthm} will then follow by Liouville's Theorem, Theorem \ref{thmMira}, which will also imply that $u \in C^\omega(\Omega)$.  Thus, $u$ is analytic if we use isothermic coordinates.
\subsection{Isothermic Coordinates} \label{secchw}
The case $n=2$ of the following version of Theorem (*) p.301 \cite{CHW} is what we'll need for the first part of the proof.
\begin{thm}(Chern-Hartman-Wintner's Existence of Isothermic Coordinates Theorem) \label{thmchw}
Let $f:\Omega \stackrel{C^n}\longrightarrow \mathbb{R}^3$ ($n \geq 2$) be a $C^n$-immersion with the induced $C^{n-1}$-metric $f^*ds^2_{\mathbb{R}^3}$ and  $C^{n-2}$ Gauss curvature $K$.  Then there exist a $C^n(\Omega)$ coordinate chart $(\beta,\Omega)$  isothermic with respect to $f^*ds^2_{\mathbb{R}^3}$ with $C^{n-1}(\beta(\Omega))$ conformal factor $h^2$.
\end{thm}
It is pointed out in \cite{CHW} that for any two charts $\beta_{(x,y)}, \beta'_{(x',y')}$, isothermic with respect to any metric $g$, the transition function $\tau = \beta' \circ \beta^{-1}$ is analytic.  In other words such mappings $\tau$ are of the form $x'+iy' = \tau(z)$ where $\tau$ is an analytic function of $z=x+iy$.  The set of such charts therefore form a $C^\omega$-atlas, which we denote by $\cal{B}$, and hence define a conformal structure making $\Omega$ a simply connected Riemann surface.  Let $\cal{A}$ denote the standard conformal structure on open subsets of $\mathbb{C}^2$.  By the uniformization theorem $(\Omega,\cal{B})$ is bi-holomorphic to $(D,\cal{A})$, for some simply connected open subset $D \subset \mathbb{C}^2$ by some map $\pi: D \longrightarrow \Omega$.  In our case, $D$ could be chosen to be the unit disk.  Then, by the Riemann mapping theorem, $(D,\cal{A})$ is bi-holomorphic to $(\Omega,\cal{A})$, by some $\rho: \Omega \longrightarrow D$.  Finally by definition, $\beta:=\rho^{-1} \circ \pi^{-1}$ is a global chart isothermic with respect to $f^*ds^2_{\mathbb{R}^3}$.  We will choose such a  global chart in the discussion below.
\subsection{Distributions and Their Derivatives} In this subsection we review distributions and their derivatives.  These will be used in the next subsection to show that our $u=\ln h$ is a weak solution to $\Delta u = e^{2u}$.  To be precise we use subscripts $x,y$ for partial differentiation and subscripts ${\cal X,Y}$ for partial distributional differentiation.

Following \cite{GT} the set of test functions for our distributions will be $C^1_0(\Omega)$,  the set of compactly supported continuously differentiable functions $v: \Omega \longrightarrow \mathbb{R}$ or $v:\Omega \longrightarrow \mbox{Mat}(2,2,\mathbb{R})$.  On $\mbox{Mat}(2,2,\mathbb{R})$ we use the inner product ${<A,B>}=\mbox{trace} A^tB$ and on $\mathbb{R}$ the inner product $<a,b> = ab$.  A distribution is defined to be a linear map $T: C^1_0(\Omega) \longrightarrow \mathbb{R}$.  Every locally integrable function $f \in L^1_{loc}(\Omega)$ defines a distribution by $T_f(v)=\int_\Omega f v$.  Such distributions are called regular and by an abuse of notation we use $f$ to denote $T_f$.  If $T$ is a distribution and $h \in C^1(\Omega)$, then $hT$ is the distribution
\[(hT)(v) = T(h^tv).\]
We define partial differentiation only for regular distributions.  If $f=T_f$ is a regular distribution, then the distributional derivatives $\partial_{\cal{X}}$ and $\partial_{\cal{Y}}$ are defined by
\begin{eqnarray*}
(\partial_{\cal{X}} f)(v) &=& -\int_\Omega <f,\partial_x v>, \\
(\partial_{\cal{Y}} f)(v) & =& -\int_\Omega <f,\partial_y v>.
\end{eqnarray*}
We will need the following standard results which are proven by straightforward calculations.
\begin{lem} \label{lemmixed}
If $W \in C^1(\Omega)$, then $W_x$ and $W_y$ are distributions and we have 
\[W_{x\cal{Y}}=W_{y\cal{X}}.\]
\end{lem}
\begin{lem}\label{lemprod}
If $P \in C^1(\Omega)$ and $L \in L^1_{loc}(\Omega)$, then $PL$ is a regular distribution and 
\begin{eqnarray*}
\partial_{\cal{X}}(PL)&=&(\partial_x P)L+P\partial_{\cal{X}}L, \\
\partial_{\cal{Y}}(PL)&=&(\partial_y P)L+P\partial_{\cal{Y}}L.
\end{eqnarray*}
\end{lem}
\begin{defn}
A differentiable function $u$ is called a weak solution to $\Delta u = g$ in $\Omega$ if 
\[\int_\Omega u_x v_x + u_y v_y  = -\int_\Omega gv \;\; \mbox{for all } v \in C^1_0(\Omega).\]
\end{defn}

\subsection{Generalized Liouville Equation} Assume we have chosen coordinates isothermic with respect to $f^*ds^2_{\mathbb{R}^3}$ so our metric is of the form $h^2(dx^2+dy^2)$.  We now show that $u=\ln h$ is a weak solution to $\Delta u = e^{2u}$.  
Let the $C^1(\Omega)$ matrix function $W$ be defined by
\begin{equation*}
W = (f_x,f_y,N), \hspace{2mm} \mbox{where} \hspace{2mm}
N = \frac{f_x}{ \mid f_x \mid} \times  \frac{f_y}{ \mid f_y \mid},
\end{equation*}
and define the $C^0(\Omega)$ matrices $A$ and $B$ by
\begin{equation*}
A = W^{-1} W_x  \hspace{2mm} \mbox{and} \hspace{2mm}
B = W^{-1} W_y,
\end{equation*}
or
\begin{equation*} \label{Lax}
W_x = W A \hspace{2mm} \mbox{and} \hspace{2mm} 
W_y = W B.
\end{equation*}
If $\ell := -\!<N_x,f_x>$, $m=-\!<N_x,f_y>=-\!<N_y,f_x>$, and $n=-\!<N_y,f_y>$.  Then
\[A= \left(\begin{array}{ccc} \frac{h_x}{h}&\frac{h_y}{h}&-\frac{\ell}{h^2} \\ -\frac{h_y}{h}&\frac{h_x}{h}&-\frac{m}{h^2} \\ \ell & m & 0  \end{array}\right)\; \mbox{and}\; B= \left(\begin{array}{ccc} \frac{h_y}{h}&-\frac{h_x}{h}&-\frac{m}{h^2} \\ \frac{h_x}{h}&\frac{h_y}{h}&-\frac{n}{h^2} \\ m & n & 0  \end{array}\right).\]
Thus $W \in C^1(\Omega),\; A,B \in L_{loc}^1(\Omega), W_x=WA \;\mbox{and}\; W_y=WB$.  By the product rule for distributions, Lemma \ref{lemprod}, we have
\[W_{x{\cal Y}} =W\!A_{\cal Y} +W_yA = W\!A_{\cal Y} + WBA.\]
Similarly $W_{y{\cal X}} =W\!B_{\cal X} +W_xB = W\!B_{\cal X} + WAB$.  Moreover, by the equality of mixed distributional derivatives, Lemma \ref{lemmixed}, we also have
\[W_{x{\cal Y}} = W_{y{\cal X}},\]
so
\[BA+A_{\cal Y} = AB+B_{\cal X}.\]
By definition, this means
\[(BA+A_{\cal Y})(v) = (AB+B_{\cal X})(v)\;\; \mbox{for all } v \in C^1_0(\Omega).\]
Or
\[A_{\cal Y}(v) - B_{\cal X}(v)=(AB-BA)(v) \;\; \mbox{for all } v \in C^1_0(\Omega). \]
Again, by definition, we have
\[-\int_\Omega (A^t v_y - B^tv_x)=\int_\Omega (AB-BA)^t v \;\; \mbox{for all } v \in C^1_0(\Omega). \]
By choosing $v=w \left(\begin{array}{ccc}0&0&0 \\ 1&0&0 \\0&0&0\end{array}\right)$, $w \in C^1_0(\Omega)$, we obtain
\[-\int_\Omega (A_{12}w_y - B_{12}w_x)=\int_\Omega (AB-BA)_{12}w \;\; \mbox{for all } w \in C^1_0(\Omega)\; \mbox{scalar}.\]
Substituting in $A$ and $B$ from above we have 
\[-\int_\Omega (\frac{h_y}{h} w_y + \frac{h_x}{h} w_x) = -\int_\Omega \frac{\ell n - m^2}{h^2} w \;\; \mbox{for all } w \in C^1_0(\Omega)\; \mbox{scalar}.\]
In our notation this is equivalent to
\[\Delta \ln{h} = -\frac{\ell n - m^2}{h^2}.\]
Since $-1 \equiv K=\frac{\ell n - m^2}{h^4}$ if follows that $\Delta \ln{h} = h^2$ and $u=\ln h$ is a weak solution of $\Delta u = e^{2u}$.
\subsection{Generalized Dirichlet Problem for Liouville Equation} \label{secGT}
\begin{thm}\label{thmGT}
If $u:\Omega \longrightarrow \mathbb{R}$, $u \in C^1(\Omega)$ and $u$ is a weak solution to $\Delta u = e^{2u}$, then $u \in C^\infty(\Omega)$.
\end{thm}
\begin{proof}
Let $p \in \Omega$ and choose a small open ball $\Omega'$ with $p \in \Omega'$.  Let $b := u|_{\partial \Omega'}$ and $a:=u|_{\overline{\Omega'}}$.  Consider the equation $\Delta w = e^{2a}$ with Dirichlet condition $w|_{\partial \Omega'} = b$.  Note that $e^{2a} \in C^1(\Omega')$ and $b \in C^1(\partial \Omega')$.  Thus Theorem 4.3 in \cite{GT} implies that $w$ exists, is unique and $w \in C^2(\Omega')$.  Now both $w$ and $u$ solve ($u$ as a weak solution) $\Delta u = e^{2u}$ on $\Omega'$ and $w \equiv u$ on $\partial \Omega'$.  Thus Theorem 8.3 in \cite{GT} implies $w \equiv u$ on $\Omega'$ also.  In particular $u \in C^2(\Omega')$.  Now we can repeatedly apply Theorem 6.17 in \cite{GT} and obtain via bootstrapping that $u \in C^\infty(\Omega')$.
\end{proof}
\subsection{Liouville's Theorem} \label{secMira}
We will use the version of Liouville's Theorem given in \cite{GM}.
\begin{thm}\label{thmMira}
$u$ solves $\Delta u = e^{2u}$ on $\Omega$ if and only if 
\[u= \frac{1}{4} \ln \frac{4 |\phi'|^2}{(1 - |\phi|^2)^2}\]
where $\phi$ is holomorphic (with respect to $z=x+i y$) with $|\phi| < 1$ and $\phi ' \neq 0$.  Furthermore the developing map $\phi$ gives an isometric immersion of $(\Omega,e^{2u} dz^2)$ into $(\mathbb{H}^2, ds^2_{\mathbb{H}^2})$.
\end{thm}

\section{Proof of Theorem \ref{mainthm}}\label{pfmain}
\begin{proof}
By Theorem \ref{thmchw} we can assume our coordinates are isothermic with respect to the induced metric
\[f^*ds^2_{\mathbb{R}^3} = h^2(dx^2+dy^2).\]
The curvature is
\[K = -1.\]
If we define $u$ by $u = \ln{h}$, then we have that $u$ is a weak solution to
\[\Delta u = e^{2u}.\]
It then follows from Theorem \ref{thmGT} that $u \in C^\infty(\Omega)$.  Our desired isometry is then guaranteed by Theorem \ref{thmMira} where $z=x+i y$.
\end{proof}
\section{Proof of Corollary \ref{maincor}} \label{seccor}
For a $C^1$-immersion $f: \Omega \longrightarrow \mathbb{R}^3$, the induced metric may be only $C^0$ with respect to the given coordinates.  Furthermore in general $K$ may not be defined.  However it was shown in \cite{DS} that if the conditions of Corollary \ref{maincor} hold, then by a change of coordinates, the immersion $f$ becomes $C^2$.  Hence in the new coordinates the induced metric is $C^1$, $K$ is $C^0$ and all the conclusions of Theorem \ref{mainthm} hold.  

More precisely we proved the following Theorem \ref{dsthm} in \cite{DS}.  We say that a $C^1$-function $f:\Omega \longrightarrow \mathbb{R}$ is $C^{1M}$ if  its mixed partials exist, are continuous, and are equal.  A $C^1$-immersion $f: \Omega \longrightarrow \mathbb{R}^3$ is $C^{1M}$ if its components are, it is asymptotic if all the parameter curves are asymptotic, and it is called Chebyshev if $<f_x,f_x> \equiv <f_y,f_y> \equiv 1$.  Here we are assuming $\theta$, the angle from $f_x$ to $f_y$, satisfies $0 < \theta < \pi$.  The subscript  ``graph" is used because the type of coordinates used are often called graph coordinates. 
\begin{thm} \label{dsthm}
Let $\Omega$ be a rectangle and $f=f_{asyche}:\Omega \stackrel{C^{1M}}{\longrightarrow} \mathbb{R}^3$ be a regular $C^{1M}$-immersion in asymptotic Chebyshev coordinates.  Assume  $N=N_{asyche}:= \frac{f_x \times f_y}{\sin{\theta}}$ is $C^{1M}$, $N_{xy}=N_{yx}=\cos \theta N$, $f_x = N \times N_x$, and $f_y=-N \times N_y$. (Note this implies  $K \equiv -1$.)   Then there exists a $C^1$-diffeomorphism $\rho:\Omega \longrightarrow \Omega$ such that $f_{graph}=f_{asyche} \circ \rho$ is a regular $C^2$-immersion.
\end{thm}

\end{document}